\documentclass[12pt]{article}
\usepackage{amsmath, amssymb, amsthm, amscd, amsfonts, cancel, enumerate, fancyhdr, geometry,  graphicx,  setspace, tikz, url, } 
\usepackage[active]{srcltx}
\usepackage[pagebackref=true]{hyperref}
\newtheorem{theorem}{Theorem}
\newtheorem*{theorem*}{Theorem}
\newtheorem{proposition}{Proposition}
\newtheorem{lemma}{Lemma}
\newtheorem{corollary}{Corollary}
\theoremstyle{remark}
\newtheorem{remark}{Remark}
\theoremstyle{definition}
\newtheorem{definition}{Definition}

\newcommand{\D}{\mathbb{D}}
\newcommand{\8}{\infty}
\newcommand{\h}{H(\D)}
\newcommand{\V}{\mathcal{V}}

\newcommand{\hh}{\mathcal{H}^2(\D)}

\newcommand{\z}{\mathcal{S}_0^2(\D)}
\newcommand{\sh}{\mathcal{S}^2(\D)}

\date{}

\begin{document}


\title{Invariant subspaces of the shift plus complex Volterra operator}

\author{\v{Z}eljko \v{C}u\v{c}kovi\'c, Bhupendra Paudyal\thanks{ This paper is based on a research which forms a part of the second author's University of Toledo Ph.D. Thesis written under the supervision of  Professor \v{Z}eljko \v{C}u\v{c}kovi\'c} }

\maketitle

\begin{abstract}
 {\small
 The lattice of closed invariant subspaces of the Volterra operator acting on $L^2(0,1)$ was completely described by Sarason \cite{Sa65}. On the other hand, he explicitly found the lattice of closed invariant subspaces of the shift plus Volterra operator on $L^2(0,1)$ in \cite{Sa74}. Inspired by Sarason's results, we find the lattice of closed invariant subspaces of the shift plus complex Volterra operator acting on the Hardy space. }  
\end{abstract}
\maketitle
\section{Introduction}
The invariant subspace problem is a well known problem in operator theory. The conjecture is that every bounded linear operator $T$ on a separable Hilbert space $H$ has a non-trivial closed invariant subspace (a closed linear subspace W of $H$ which is different from $\{0\}$ and $H$ such that $T(W) \subset W$).
Over the years it has been shown that many familiar classes of operators do have invariant subspaces. The goal of this paper is to characterize the lattice of closed invariant subspaces of the shift plus complex Volterra operator on the Hardy space. 
 
Let $\D$ be the unit disk of the complex plane and $\h$ be the space of holomorphic functions on the unit disk. We say that a holomorphic function $f(z)=\sum_{n=0}^\8 a_n z^n$ on the unit disk belongs to the Hardy space $\hh$, if its sequence of power series coefficients is square-summable: 
\[\hh=\{f\in\h:\sum_{n=0}^\8|a_n|^2<\8\}.\]
We define a norm on $\hh$ by 
\begin{equation} \label{HN1}
\|f\|^2_{\hh}= \sum_{n=0}^\8|a_n|^2.
\end{equation} 
It is well known that $\hh$ is a Hilbert space with the inner product \[ \langle f,g\rangle_{\hh} =\sum _{n=0}^\8 a_n \overline{b_n}\] for $f(z)= \sum_{n=0}^\8 a_n z^n$ and $g(z) =\sum_{n=0}^\8 b_n z^n$. 

  The operator defined on $\hh$ \[(M_zf)(z)= zf(z)\hspace{.5cm} \text{for}~ f\in\hh~\text{and}~ z\in\D\] is called  the shift operator.  The lattice of the shift operator acting on the Hardy space is completely described by Beurling's Theorem \cite{Be48}, and it is one of the  the most celebrated and widely used results. Let $L^2(0,1)$ be the space of square integrable functions on $(0,1)$. Sarason \cite{Sa65} characterized all closed invariant subspaces of the Volterra operator \[(Vf)(x)=\int_0^x f(y) dy\hspace{.5cm}\text{for}~ f \in L^2(0,1)~\text{and}~0< x\leq y< 1.\]  Aleman and Korenblum studied the complex Volterra operator in $\hh$ defined by \[(\V f)(z)=\int_0^z f(w) dw.\] Then they characterized the lattice of closed invariant subspaces of $\V$ in \cite{Al08}. While doing so they used the Beurling's Theorem. On the other hand, Sarason \cite{Sa74} studied the lattice of closed invariant subspaces of multiplication by $x$ plus Volterra operator, $M_x+V$ acting on $L^2(0,1)$. Montes-Rodriguez, Ponce-Escudero and Shkarin \cite{MO10} and Cowen, Gunatillake and Ko \cite{CO13} used the idea of Sarason to study the invariant subspaces of certain classes of composition operators on Hardy spaces.
  
  Following Sarason's work we are interested in characterizing the lattice of closed invariant subspaces of the shift plus complex Volterra operator on the Hardy space. Denote by $T$ the operator
\begin{equation}
(Tf)(z)=z f(z)+\int_0^z f(w) dw, \hspace{.5cm} \text{for}~ f\in \hh ~\text{and}~ z\in \D.
\end{equation}
Since the shift operator is an isometry and the complex Volterra operator is a contraction, $T$ is clearly bounded operator on the Hardy space. To prove our main result we use the space $\sh$ defined by 
 \[\sh =\{f\in \h: Df\in \hh\},\] where $D=\frac{d}{dz}$ is the differential operator. It is clear that if $Df$ is in  $\hh$, then $f$ belongs to $\hh$. The norm of $\sh$ is defined by  \begin{equation} \label{IN1}
 \|f\|^2_{\sh} = \|Df\|^2_{\hh }+\|f\|^2_{\hh}.
\end{equation} 
Corresponding inner product is given by \[\langle f,g\rangle_{\sh}=\langle Df,Dg\rangle_{\hh}+\langle f,g\rangle_{\hh}.\]  
Our work describes the lattice of closed invariant subspaces of $T$ acting on $\hh$. The main tool in our proof is Korenbljum's result from \cite{KO72}, in which he characterized all closed ideals of $\sh$. Korenbljum's paper, in turn, relates to the work of Rudin \cite{RU57} and his main theorem on the structure of the closed ideals of the disk algebra. The following result is known to the specialists but we could not find the proof in the literature. We include the proof for the sake of completeness.
\begin{proposition}\label{INP1}
The following statements are true:
\renewcommand\theenumi{\roman{enumi}}
\begin{enumerate}
\item $\sh \subset H^\8.$
\item $\sh$ is Banach algebra.
\item Polynomials are dense in $\sh$.
\end{enumerate}
\end{proposition}
\begin{proof}
\renewcommand\theenumi{\roman{enumi}}
\begin{enumerate}
 \item 
 Let $f \in \sh$, then $Df\in \hh$  and hence \[|Df(z)|\leq \dfrac{\|Df\|_{\hh} }{\sqrt{1-|z|^2}},~\hspace{.5cm}~\text{for~ all}~z\in\D.\] Now,
 \begin{align*}
| f(z)-f(0)|
 =&|\int_0^1 zDf(tz) dt| \\
  \leq &\int_0^1 |zDf(tz)|dt \\ 
 \leq &\|Df\|_{\hh} \int_0^1 \frac{|z|}{\sqrt{1-|tz|^2}} dt \\
 \leq &\|Df\|_{\hh} \int_0^1 \frac{|z|}{\sqrt{(1-|tz|)(1+t|z|)}} dt. 
 \end{align*}
 Since $1+t|z|\geq 1$, we have $\dfrac{1}{\sqrt{1+t|z|}}\leq 1$. 
 \begin{align} \label{IEqN1}
| f(z)-f(0)| \leq &\|Df\|_{\hh} \int_0^1 \frac{|z|}{\sqrt{1-|tz|}} dt \nonumber\\
\leq & 2 \|Df\|_{\hh}.
\end{align}
This clearly shows that $f$ belongs to $H^\8$, and hence $\sh\subset H^\8.$
\item To show $\sh$ is a Banach space under the above norm, suppose $\{g_n\}_0^\8$ is Cauchy in $\sh$ norm. Then, clearly $\{g_n\}$ and $\{Dg_n\}$ are Cauchy on $\hh$ norm. Since $\hh$ is a Banach space, $\{g_n\} $ converges to a holomorphic function $g\in \hh$ and $\{Dg_n\}$ converges to the function $Dg \in \hh$. Therefore $g\in\sh$, and hence $\sh$ is a Banach space under the given norm. Pointwise multiplication on $\sh$ form an algebra. For this, suppose that $f$ and $g$ are in $\sh$. 
\begin{align}\label{IEQN1}
\|fg\|^2_{\sh} =& \|D(fg)\|^2_{\hh }+\|fg\|^2_{\hh}\nonumber\\
=&\|gDf+fDg\|^2_{\hh}+\|fg\|^2_{\hh}\nonumber\\
\leq&\|gDf\|^2_{\hh}+2 \|gDf\|_{\hh}\|fDg\|_{\hh}+\|fDg\|^2_{\hh}+ \|fg\|^2_{\hh}.
\end{align}
Using \eqref{IEqN1}, we see that for any $f\in\sh$ 
\begin{align*}
\|f\|_\8\leq & 2\|Df\|_{\hh}+|f(0)|\\
\leq & 2\|Df\|_{\hh}+ 2\|f\|_{\hh}=2\|f\|_{\sh}.
\end{align*}
\text{Hence using \eqref{IEQN1}},
\begin{align*}
\|fg\|^2_{\sh}\leq&\|g\|^2_\8 \|Df\|^2_{\hh}+2 \|f\|_\8\|g\|_\8\|Df\|_{\hh}\|Dg\|_{\hh} \\ &+\|f\|^2_\8\|Dg\|^2_{\hh} + \|g\|^2_\8\|f\|^2_{\8} \\
\leq & 4\|g\|^2_{\sh} \|f\|^2_{\sh}+ 8\|g\|^2_{\sh}\|f\|^2_{\sh} + 4\|g\|^2_{\sh} \|f\|^2_{\sh}\\
\leq & 16 \|f\|^2_{\sh}\|g\|^2_{\sh}.\end{align*}
\begin{remark}
Clearly $\sh$ is an algebra with unit and with a norm $\|.\|_{\sh}$ under which it is a Banach space. Furthermore, we see that the multiplication is continuous in each factor separately. Therefore there exists a norm equivalent to $\|.\|_{\sh}$, for which $\sh$ is a Banach algebra (see \cite[page 212]{Ka04}). Keeping in mind the fact above, we do not differentiate the norm $\|.\|_{\sh}$ and its equivalent norm that makes $\sh$ a Banach algebra.
\end{remark}
 \item We want to show polynomials are dense in $\sh$.
 Given $f\in\sh$, let $f_q(z)=f(qz)$ be its dilation  and $Df_q(z)= qDf(qz)=q(Df)_q(z)$ be derivative of dilation where $0<q<1$. Each function $f_q$ is analytic in a larger disk, so it can be approximated uniformly on $\D$ by a sequence of holomorphic polynomials $P_q^n$, and hence $Df_q$ can be approximated uniformly on $\D$ by holomorphic polynomials $DP_q^n$. 
 So it will be enough to prove that $f$ can be approximated in $\sh$ by its dilation. That is to say $\|f-f_q\|_{\sh}\rightarrow 0$ as $ q\rightarrow 1$. This means that $\|f-f_q\|^2_{\hh}+\|Df-Df_q\|^2_{\hh}\rightarrow 0$ as $ q\rightarrow 1$. So finally it is enough to prove that $\|f-f_q\|^2_{\hh}\rightarrow 0$ and $\|Df-Df_q\|^2_{\hh}\rightarrow 0$ as $ q\rightarrow 1$. For this, let us assume \[\displaystyle{f(z)= \sum _{n=0}^\8 a_n z^n}.\] 
 Since $f\in\hh$, for all $\epsilon >0$,  we can choose a natural number $N$ large enough such that  
  \[\displaystyle{\sum _{n=N+1}^\8 |a_n|^2<\frac{\epsilon}{2}}.\]  
  Now choose $q_\epsilon\in (0,1)$ such that \[(1-q_\epsilon^N)^2\sum_{n=0}^N|a_n|^2<\frac{\epsilon}{2}.\]   
 Then, since \[\|f-f_q\|^2_{\hh}=\left\|\sum_{n=0}^\8 a_n z^n (1-q^n)\right\|^2=\displaystyle{\sum_{n=0}^\8 |a_n(1-q^n)|^2},\]
 it follows that for all $ q\geq q_\epsilon$
  \begin{align*}  
  \|f-f_q\|^2_{\hh}= 
  =&\sum_{n=0}^N |a_n(1-q^n)|^2+\sum_{n=N+1}^\8 |a_n(1-q^n)|^2\\
  \leq& (1-q^N)\sum_{n=0}^N |a_n|^2 + \sum _{n=N+1}^\8 |a_n|^2 \\
  \leq & \frac{\epsilon}{2}+\frac{\epsilon}{2}=\epsilon. 
  \end{align*} This shows that $\|f-f_q\|^2_{\hh}\rightarrow 0$ as $q\rightarrow 1$. On the other hand, we have
  \begin{align*} 
 |Df(z)-Df_q(z)|=&\left|\sum_{n=1}^\8na_nz^{n-1}-\sum_{n=1}^\8na_nq^nz^{n-1}\right|=&\left|\sum_{n=1}^\8na_n(1-q^n)z^{n-1}\right|.
\end{align*}  
Similarly, we can  show that $\|Df-Df_q\|^2\rightarrow 0$ as $q$ approaches to $1$. 
\end{enumerate}
\end{proof}
\begin{definition} \label{ID1}
Define \[\z=\{f\in\sh : f(0)=0\}\]
\end{definition}
\begin{corollary} \label{ICO1}
$\z\subset \sh $ is a Banach algebra with the norm defined for $\sh$ and $\sh=[1]\bigoplus \z$, and hence \[\z= \overline{span}\{z^n:n\in \mathbb{N}\}.\]
\end{corollary}
\begin{proof} For any $f$ and $g$ in $\z\subset \sh$, we immediately see that \[\|fg\|_{\sh}\leq\|f\|_{\sh}\|g\|_{\sh}.\]Also, we have $f(0)=0$ and $g(0)=0$ so $(fg)(0)=0$ and  hence $fg$ belongs to $\z$.
To show $\z$ is a closed subalgebra of $\sh$, assume $g\in\overline{\z}$. That means there exists a sequence $g_n\in\z$, $n\in\mathbb{N}$ such that $g_n$ converges to $g$ in $\sh$ norm. This implies $g_n$ converges to $g$ in $\hh$ norm. Since $g_n(0)=0$ for all $n$, it follows that $g(0)=0$.  Moreover, the other part of the corollary follows from the Definition \ref{ID1}.   
\end{proof}
The following theorem plays important role in the proof of main theorem.
\begin{theorem}
Let $\V$ be the Volterra operator on $\hh$. Then the following statements are true:
\renewcommand\theenumi{\roman{enumi}}
\begin{enumerate}
\item  Range of $\V= \z$.
\item $\V$ is a bounded isomorphism from $\hh$ onto $\z$, and its inverse is $D$.
\item The operator $T$ acting on $\hh$ is similar under $V$ to the multiplication operator $M_z$ acting on $\z$. 
\end{enumerate}
\end{theorem}
\begin{proof}
\renewcommand\theenumi{\roman{enumi}}
\begin{enumerate}
\item Let $g$ be in the range of $\V$, then there exists $ f\in\hh$ such that 
\[\displaystyle{g(z)= (\V f)(z)=\int_0^z f(w)dw}\] 
then  $Dg=f\in \hh$ and $g(0)=0$. Hence $g\in \z$.
Conversely, suppose that $g$ belongs to $ \z$, then 
\[\displaystyle{(\V Dg)(z) = \int_0^z (Dg)(w)dw =g(z)-g(0)=g(z)}.\]
 Therefore $g$ belongs to the range of $\V$. 
\item First we want to show $\V$ is a bounded operator on $\hh$. Let us assume $f$ is in the Hardy space. 
 \begin{align*}
\|\V (f)\|_{\sh}=&\|\V (f)\|_{\hh}+\|D(\V (f))\|_{\hh}\\
\leq &\|f\|_{\hh} +\|f\|_{\hh}\\=&2\|f\|_{\hh}.
\end{align*} 
Hence the map $\V$ from $\hh$ onto $\z$ is bounded. Clearly $\V$ is linear. Now to show $\V$ is one-one, assume that $f_1$ and  $f_2$ belong the Hardy space, and also assume that  \[ \int_0^z f_1(w)dw=\int_0^z f_2(w)dw, \hspace{1cm} \text{for~all}~ z \in\D.\] Differentiating both sides we see that $f_1=f_2$ and hence $\V$ is one-one. From part $(i)$ we have $\V Dg=g$ and clearly $D\V f=f$. This shows that $\V$ is a bounded bijective linear operator from $\hh$ onto $\z$ and $\V^{-1}=D$.
\item Suppose $f$ belongs to $\hh$ and also suppose $\V f =g$, for some $g\in \z$. Therefore we have $f(z)=(\V ^{-1}g)(z)= (Dg)(z)$. 
\begin{align*}
(Tf)(z)
=&zf(z)+(\V f)(z)\\
=&z(Dg)(z)+g(z)\\
=&D(zg(z)).\end{align*}
Now applying $\V$ on the both side, we see that
\begin{align*}
(\V Tf)(z)=&\V D(zg(z))\\=&zg(z)\\=&z(\V f)(z).
\end{align*}\end{enumerate}
So, $\V T=M_z\V$ and $\V T\V^{-1}=M_z$. 
 That is to say $\V$ transforms the operator $T$ into the operator multiplication by $z$ on $\z$.
 \end{proof}
 
 We can summarize the theorem by the following commutative diagram
\begin{center}
\begin{tikzpicture}
\coordinate[label=above :$\z$](A) at (6,-.2);
\coordinate[label=below:$\z$] (B) at (6,4.2);
\coordinate[label= below:$\hh$] (C) at (0,4.2);
\coordinate[label=above:$\hh$] (D) at (0,-.2);
\draw[thick,->] (6,3.2)--(6,0.8);
\draw[thick,->] (0.8,4)--(5.2,4);
\draw[thick,->] (0,3.2)--(0,0.8);
\draw[thick,->] (0.8,0)--(5.2,0);
\draw[thick,->] (.8,3.2)--(5.2,.8);
\draw (6,0.8) -- (6,3.2) node[right,midway]{$M_z$ };
\draw (5.2,4) -- (0.8,4) node[above,midway]{$\V$ };
\draw (0.8,0)--(5.2,0) node[below,midway] {$ \V$ };
\draw (0,3.2)--(0,0.8) node[left,midway] {$T$ };
\draw (.8,3.2)--(5.2,.8) node[above,midway] {$\V T=M_z\V$ };
\end{tikzpicture}
\end{center}

\begin{definition} An element $a$ in Banach algebra $\mathcal{A}$ is called cyclic if the subalgebra generated by $a$ is dense in $\mathcal{A}$. 
\end{definition}

The next proposition provides the relation between closed invariant subspaces of multiplication operator and closed ideals of a Banach algebra. Proof of this proposition can be found in \cite{MO10}.
\begin{proposition} \label{IPR2}  Let $\mathcal{A}$ be a Banach algebra. Then the invariant subspaces of multiplication by a cyclic element are exactly the closed ideals of $\mathcal{A}$.
\end{proposition}

\begin{lemma} \label{ILe1}
$J$ is a closed ideal of $\z$ if and only if $J$ is an ideal of $\sh$ contained in $\z$.
\end{lemma}
\begin{proof}
Using of Corollary \ref{ICO1}, we see that for any $h\in \sh$ there exists $h_1\in\z$ such that $h=c+h_1$.

 Suppose $J$ is a closed ideal of $\z$, then
 for any $h\in\sh$ and $j\in J$,
\[hj =(c+h_1)j=cj+h_1j \in J.\]
Since norm on both spaces are the same, $J$ is a closed ideal of $\sh$.
 On the other hand, if $J$ is an ideal of $\sh$ contained in $\z$, then it is clear that $J$ is a closed ideal of $\z$.
\end{proof}
  
 \begin{definition}\label{def3}
 Let $K$ be a closed subset of the unit circle $\partial\D$. For and inner function $G$ we say $G$ is \textit{associated with} $K$ if 
\begin{itemize}
\item if $a_1,a_2,...$ are the zeros of $G(z)$ in the open disk, then all the limit points of $\{a_k\}$ belong to $K$;
\item the measure determining the singular part of $G$ is supported on $K$.
\end{itemize}
 \end{definition}
 
 The following theorem from Korenbljum \cite{KO72} characterizes all closed ideals of $\sh$. This theorem is essential in proving our main theorem.
 \begin{theorem} \label{KOth11}  Suppose $K$ is a is closed subset of $\partial \D$ and let $G$ be an inner function associated with $K$. Let $I_{\sh}(G;K)$ be the set of all  $f\in\sh$ which are divisible by $G$ and which vanish on $K$. Then $I_{\sh}(G;K)$ is a closed ideal of $\sh$. Moreover, every closed ideal of $\sh$ is obtained in this manner.
\end{theorem}

\begin{corollary}
\label{cor1}
 Suppose $K$ is a is closed subset of $\partial \D$ and let $G$ be an inner function associated with $K$.
Let $I_{\z}(G;K)$ be the set of all  $f\in\z$ which are divisible by $G$ and which vanish on $K$. Then $I_{\z}(G;K)$ is a closed ideal of $\z$. Moreover, every closed ideal of $\z$ is obtained in this manner.

\end{corollary}
\begin{proof}
Clearly $I_{\z}(G;K)$ is a closed ideal of $\z$. 
Conversely, suppose $J$ be any ideal of $\z$. From the Lemma \ref{ILe1}, $J$ is also an ideal of $\sh$ contained in $\z$. By the Theorem \ref{KOth11}, $J$ is of the form $I_{\sh}(G;K)$  for some $G$ and $K$ defined in Definition 3. Hence $J=I_{\z}(G;K)$.
\end{proof}

\begin{theorem} [{\bf Main Theorem}]
\label{Manin Theorem}
Let $T$ be an operator 
\[ (Tf)(z) =zf(z)+ \int_0^z f(w) dw.\] defined on $\hh$. Then the lattice of closed invariant subspaces is\\
\[Lat ~T=\{S \subset \hh :S=\{Df:f \in I_{\z}(G;K)\}~\text{for}~G, K~\text{defined~in~Definition \ref{def3}} \}.\] 
\end{theorem}

\begin{proof}
From Corollary \ref{ICO1}, \[\z= \overline{span}\{z^n:n\in \mathbb{N}\},\] so $z$ is a cyclic element of the Banach algebra $\z$. Thus from the Proposition \ref{IPR2}, closed invariant subspaces of $M_z$ on $\z$ are exactly the closed ideals of $\z$. Using Corollary \ref{cor1}, the lattice of closed invariant subspace of $M_z$ acting on $\z$ is given by
\[Lat~ M_z=\{I_{\z}(G;K): G, K~ \text{defined~in~Definition \ref{def3}}\}.\]
Since $\V T\V^{-1}=M_z$, we see that
$\V^{-1} (I_{\z}(G;K))$ is a closed invariant subspace of $T$. From Theorem 1, we know that $\V^{-1} (f)=Df$. So, $S=\{Df:f \in I_{\z}(G;K)\}$ is a closed invariant subspace of $T$. Hence,
\[Lat ~T=\{S \subset \hh :S=\{Df:f \in I_{\z}(G;K)\}~\text{for}~G, K~\text{defined~in~Definition \ref{def3}} \}.\] 

\end{proof} 
\begin{remark}
One can notice that if $K$ is of Lebesgue arc length measure zero, then the ideal $I_{\z}(G;K)\}$ is the zero ideal.
\end{remark}

Department of Mathematics and Statistics, The University of Toledo, Toledo, Ohio,
{\it E-mail address: zeljko.cuckovic@utoledo.edu}\\
Department of Mathematics and Statistics, The University of Toledo, Toledo, Ohio,
{\it Email address: bpaudya@rockets.utoledo.edu}


\begin{thebibliography}{56}
 \bibitem{Al08} A. Aleman and B. Korenblum, \emph{Volterra invariant subspaces of $H^p$,} Bull. Sci. Math. 132 (2008).
 
\bibitem{Be48} A. Beurling, \emph{On two problems concerning linear transformations in Hilbert space,} Acta Math. 81, (1948).

\bibitem{CO13} C. C. Cowen, G. Gunatillake and E. Ko, \emph{Hermitian weighted composition operators and Bergman extremal functions}, Complex Anal. Oper. Theory 7 (2013), no. 1, 69–99. 

\bibitem{Ka04} Y. Katznelson, \emph{An introduction to harmonic analysis}, Third edition, Cambridge Mathematical Library, Cambridge University Press, Cambridge, (2004).

 \bibitem{KO72} B. I. Korenbljum, \emph{Invariant subspaces of the shift operator in a weighted Hilbert space,} (Russian) Mat. Sb. (N.S.) 89(131) (1972).
 
\bibitem{MO10} A. Montes-Rodriguez, M. Ponce-Escudero and S. A. Shkarin, \emph{Invariant subspaces of parabolic self-maps in the Hardy space,} Math. Res. Lett. 17 (2010).
 
\bibitem{RU57} W. Rudin, \emph{The closed ideals in an algebra of analytic functions}, Canad. J. Math. 9 (1957), 426-434. 

 \bibitem{Sa65} D. Sarason, \emph{A remark on the Volterra operator,} J. Math. Anal. Appl. 12 (1965).
 
 \bibitem{Sa74} D. Sarason, \emph{Invariant subspaces,} Topics in operator theory, Math. Surveys, No. 13, Amer. Math. Soc., Providence, R.I., (1974).
\end{thebibliography}
\end{document}